\documentclass[12pt]{amsart}
\usepackage{amssymb}
\usepackage{graphicx}
\usepackage{colordvi}
\usepackage{youngtab}
\usepackage{tikz}
\usepackage{apacite}

\numberwithin{equation}{section}
\newtheorem{theorem}{Theorem}[section]
\newtheorem{proposition}[theorem]{Proposition}
\newtheorem{lemma}[theorem]{Lemma}
\newtheorem{corollary}[theorem]{Corollary}
\theoremstyle{remark}
\newtheorem{example}[theorem]{Example}
\newtheorem{remark}[theorem]{Remark}
\newtheorem{defn}[theorem]{Definition}

\newcounter{FNC}[page]
\def\fauxfootnote#1{{\addtocounter{FNC}{2}$^\fnsymbol{FNC}$%
     \let\thefootnote\relax\footnotetext{$^\fnsymbol{FNC}$#1}}}
\newcommand{\C}{{\mathbb{C}}}
\newcommand{\N}{{\mathbb{N}}}
\newcommand{\Z}{{\mathbb{Z}}}
\newcommand{\R}{{\mathbb{R}}}

\newcommand{\K}{{\mathbb{K}}}

\title{On the Toric Ideals of coloured graphs of reduced words}

\author{Praise Adeyemo}
\address{Department of Mathematics\\
  University of Ibadan\\
   Ibadan, Oyo, Nigeria}
\email{ph.adeyemo@ui.edu.ng, praise.adeyemo13@gmail.com}
\urladdr{http://sci.ui.edu.ng/HPAdeyemo}


\copyrightinfo{}{}
\begin{document}
\begin{abstract}
 We study a family $\mathcal{B}$ of pseudo-multipartite  graphs  indexed by staircase partitions. They are realised from the reduced words of certain class of permutations. We investigate the vertex proper colouring  of these graphs  and give the general chromatic polynomial. For each member $B_{\lambda}$,  we construct an affine toric  ideal $\mathcal{I}_{B_{\lambda}}$ associated to vertex proper colouring  using partition identity.  It turns out that the projective version  $\mathcal{V}(\mathcal{I}_{B_{\lambda}})$ is realised from the cartoon diagram associated with the vertex proper colouring.
\end{abstract}

\maketitle
%%%%%%%%%%%%%%%%%%%%%%%%%%%%%%%%%%%%%%%%%%%%%%%%%%%%%%%%%%%%%%%%%%%%%%%%%%%%%%%%%
%
\section{introduction}
\noindent In this work we study  vertex proper colouring of pseudo-multipartite graphs corresponding to the checkerboard  colouring of Ferrers diagrams  associated with staircase partitions. This is closely related to the partition colouring studied in [4], [11], [12]. These graphs are given by imposing a particular orientation on the graphs of reduced words of certain permutations in union of symmetric groups $\mathfrak{S}_r$ for $r\geq 4$. The goal is the colouring of the members of this family of graphs. Every member of the family is bipartite and therefore they are bicoloured. It turns out that this bicolouring is the specialisation $k=2$ in the diagonal $k$-coloured partitions studied in [7] with a slight difference in labellings.  Specifically, we obtain defining equations of the binomial ideal associated to each of the multipartite graphs and give two conjectures along this direction.  In recent years, there have been extensive studies on the reduced words of permutations, [1], [7], [9].  A graph $(V,E)$ is said to be bipartite denoted by $B_{r,s}$ if there exist subsets $V_1$ and $V_2$ of the vertex set $V$ such that $V:=\bigsqcup_{k=1}^{2} V_{k}$ and each \ edge\ in  $E$  is  of  the  form $vw$  with  $v\in V_1$  and  $w\in V_2$.

For instance, the bipartite graph $B_{5,3}$ is given in Figure 1.

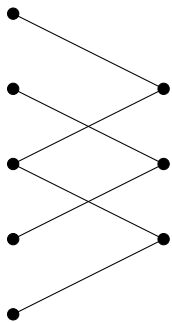
\begin{figure}[!hbt]
\begin{center}
\begin{tikzpicture}
\draw(0,0)--(2,1)--(0,2);
\draw (0,1)--(2,2)--(0,3);
\draw(0,4)--(2,3)--(0,2);
\draw[fill] (0,0) circle [radius=0.075];
\draw[fill] (0,1) circle [radius=0.075];
\draw[fill] (0,2) circle [radius=0.075];
\draw[fill] (2,1) circle [radius=0.075];
\draw[fill] (0,3) circle [radius=0.075];
\draw[fill] (0,4) circle [radius=0.075];
\draw[fill] (2,2) circle [radius=0.075];
\draw[fill] (2,3) circle [radius=0.075];

\end{tikzpicture}
\end{center}
\caption{{\rm The \ bipartite \ graph} \ $ B_{5,3}$}
\end{figure}
\noindent We review the relevant background in section two and give some preliminary results. The general chromatic polynomial of the the proper vertex colouring is given in section three while the associated toric ideals are constructed in section four. 
\section{Graphs of Reduced words of Permutations}
\noindent{\bf Partitions of Integers:}  A partition $\lambda$ of $n\in\N$ denoted $\lambda\vdash n$  is a list $\lambda=(\lambda_1\geq\lambda_2\geq\cdots\geq\lambda_k)$ such that $\lambda_1+\lambda_2+\dots +\lambda_k= |\lambda|=n$. The number $k=\ell (\lambda)$ is the length of the partition and each $\lambda_{i}$ is called a part of the partition $\lambda$.  Associated to every partition $\lambda\vdash n$  is its transpose, $\lambda^{t}:=(\lambda^{t}_{1},\dots,\lambda^{t}_{m})$, which is also a partition of $n$ where $\lambda^{t}_{i}$ is the number  parts of $\lambda$ which are at least $i$. Any partition that coincides with its transpose is said to be self conjugate. To each partition $\lambda\vdash$ of length $k$ we associate a diagram $Y(\lambda)$ called  Ferrers diagram  consisting of $|\lambda|$ boxes having $k$ left -justified rows with row $i$ containing $\lambda_{k+1-i}$ boxes for $1\leq i\leq k$. We  now discuss a class of partitions  that we care about.  A partition $\lambda\vdash n$ is said to be staircase if it is of the form $(r,r-1,r-2,\dots, 2,1)$  and every staircase partition is self conjugate.
 The Ferrers diagram associated to the staircase partition $\lambda = 6, 5,4,3,2,1$  is given below.
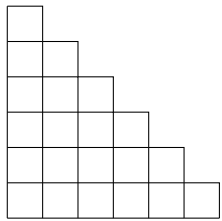
\begin{center}
\begin{tikzpicture}[scale=0.47]
			\draw (0,0)--(5,0);
			\draw (0,1)--(5,1);
			\draw (5,1)--(5,0);
			\draw (0,0)--(0,5);
			\draw (1,0)--(1,5);
			\draw (0,5)--(1,5);
			\draw (0,2)--(4,2);
			\draw (4,2)--(4,0);
			\draw (0,3)--(3,3)--(3,0);
			\draw (0,4)--(2,4)--(2,0);
			\draw (6,0)--(6,1)--(5,1);
			\draw (5,1)--(5,2)--(4,2);
			\draw (4,2)--(4,3)--(3,3);
			\draw (3,3)--(3,4)--(2,4);
			\draw (2,4)--(2,5)--(1,5);
			\draw (1,5)--(1,6)--(0,6);
			\draw (0,5)--(0,6);
			\draw (5,0)--(6,0);
		        \end{tikzpicture}
\end{center}
\noindent Notice that not every positive integer $n$ admits a staircase partition. The characterisation  of positive integers which encode staircase partitions is the following proposition. 

\begin{proposition}
Let $n$ be a positive integer. Then  there exits a unique staircase partition $\lambda\vdash n$ if and only if $n$ is a triangular number. Furthermore,  $\lambda$ is the longest distinct parts partition of $n$.
\end{proposition}
\begin{proof}
 The set of triangular numbers is well-ordered and the sequence is given by $\{{{k+1}\choose 2}\}_{k=1}$. Denote this by $\mathcal{X}$ and suppose that  the term $n$ is located at the position $r$  of the sequence  $\mathcal{X}$, then $n$ is uniquely expressed  as  $r+(r-1)+\cdots + 1$, since every term in the sequence has a unique position so $n$ admits the unique staircase partition $\lambda=(r,r-1,\dots, 2,1)$. Conversely, suppose that staircase $\lambda=(\lambda_1,\lambda_2,\dots,\lambda_{\ell(\lambda)})$  is a staircase partition,  the sum $|\lambda|=\sum_{i=1}^{\ell(\lambda)}\lambda_{i}$ is the $\ell(\lambda)^{\rm th}$ term in the sequence $\mathcal{X}$. Therefore $|\lambda|$ is a triangular number.  Mow. by definition,  $\lambda$ consists of positive consecutive numbers starting from one which sum up to $n$. Therefore, it is easily the longest among the distinct part partitions of $n$. 
\end{proof}

\noindent It turns out that the sequence of triangular numbers  is given by a generating function which corresponds to certain well known product expansions.
\begin{equation}
\sum_{r\geq 0}^{\infty} {\rm P}(r) z^{r} =\frac{z}{(1-z)^3}
\end{equation}
where ${\rm P}(r)$ is the triangular number which admits the staircase partition $\lambda$ of length $r$.\\
\ \\
\noindent{\bf Symmetric groups:} Let $\frak{S}_n$ denote the symmetric group of permutations on the set $[n] :=\{1,\dots,n\}$. A permutation is a bijective self map on $[n]$. One-line notation shall be adopted to represent the elements of $\mathfrak{S}_n$, that is, if $w\in \frak{S}_n$ then $w=w_{1}w_{2}\cdots w_{n}$ such that $w(i)=w_{i}$. The product of $w$ and $v$ in $\frak{S}_{n}$ is the composition wv as functions. The length $\ell(w)$ of $w$ is $\#\{(i,j)\ : \ w(i)>w(j), 1\leq i<j\leq n\}$, the number of inversions in $w$.  A permutation $s_{i}:=(i,i+1)\in \frak{S}_n$, $i\in\{1,2,\dots,n-1\}$ which swaps the letters $i$ and $i+1$ keeping other letters fixed is called a simple transposition. It is well known that the symmetric group $S_{n}$ is finitely presented by simple transpositions $s_1,s_2,\dots,s_{n-1}$ and  the  relations are as follows.
\begin{center}
$s_{i}^{2} =e,    \ \ \ \ \ \ \  {\rm for} \ \ \ \ 1\leq i\leq n-1$\\
\ \ $s_{i}s_{i+1}s_{i} = s_{i+1}s_{i}s_{i+1}, \ \ \ \ \ {\rm for} \ \ \ \ \ 1\leq i\leq n-2$ \ \ \ \ \ \  {\rm (Braid)}\\
\ \ \ \ \ \ \ \ \ \ \ \  \ \ \ $s_{i}s_{j} = s_{j}s_{i}, \ \ \ \  \ \ \ \ \ \ \ \ {\rm for} \ \ \ \ \mid i-j\mid > 1$, \ \ \ \ \ \  {\rm (commutation)}
\end{center}

\noindent If the product $s_{a_{1}}s_{a_{2}}\cdots s_{a_{r}} = w$ such that $r=\ell(w)$ then the sequence ${\bf a} = a_{1}a_{2}\cdots a_{r}$ is called a reduced word for $w$. In fact, a reduced word can be identified with a minimal sequence of generators whose product is $w$ since $w\ne s_{a_{1}}s_{a_{2}}\cdots s_{a_{r}}$ for $r<\ell(w)$. The reduced expression $s_{a_{1}}s_{a_{2}}\cdots s_{a_{r}}$ for $w$ is not unique and as a matter of fact, one reduced expression of $w$ can be changed to another by a sequence of local transformations described in the last two relations above called braid relation and commutation relation respectively. This observation gives rise to the graph $\mathcal{G}_{w}$ of reduced words of the permutation $w$ in which the reduced words of $w$ constitute the vertices of the graph such that there is an edge between any pair of reduced words if there is either a braid or commutation relation between them. We denote the set of reduced words for $w\in S_{n}$ by $R(w)$  and its cardinality by $r(w)$.  For instance, if $w = 35124$ then
\begin{center}
$R(w)=\{42312, 24312, 42132, 24132,21432 \} \ \ {\rm and} \ \ r(w)=5$
\end{center}
The graph $\mathcal{G}_{35124}$  of the reduced words of the permutation $w=35124$  is given below

\begin{figure}[!hbt]
	\begin{center}
	  \begin{tikzpicture}[x= 1.7cm, y=1.7cm]
	   \tikzset{vertex/.style = {shape=circle,draw,minimum size=10.5em}}
	    \tikzset{edge/.style = {-,> = latex'}}
		\node (a) at (0,3) {$42312$};
		\node (b) at (0,2) {$24312$};
		\node (c) at (2,2) {$24132$};
		\node (d) at (2,3) {$42132$};
		\node (e) at (2,1) {$21432$};
		\draw (e)--(c)--(b)--(a)--(d)--(c);
		\end{tikzpicture}
	\end{center}
	\caption{The graph $\mathcal{G}_{35124}$ of the reduced words of $w=35124$}
\end{figure}
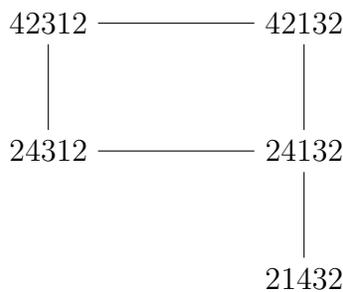  
\noindent We are interested in the family $\mathcal{B}$ of graphs of reduced words specified by a certain subcollection $\mathcal{Z}$ of permutations $\sigma$ in  the union  $\cup_{r=4}^{\infty}\frak{S}_{r}$ . These are given by
\begin{equation}
\mathcal{Z}=\{\sigma\in\cup_{r=4}\mathfrak{S}_{r} : \  \sigma = 2 ,\ 3,\cdots r-3, r ,\  r-2 ,\  r-1,  1\}
\end{equation}

\noindent For instance, if $r=4$, then $\sigma = 4231$. There are six reduced words  associated with $\sigma$, namely:  $$32123, 31213, 13213, 31231, 13231, 12321$$ and the graph $\mathcal{G}_{4231}$ is illustrated in Figure 3.  

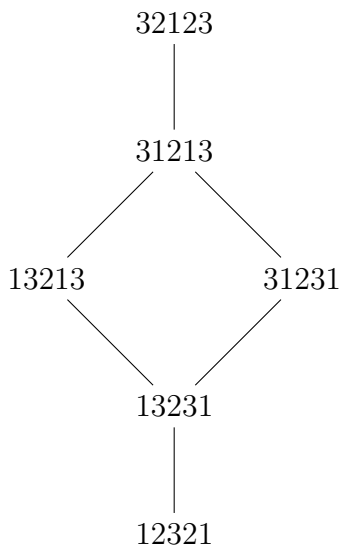
\begin{figure}[!hbt]
	\begin{center}
	  \begin{tikzpicture}[x= 1.7cm, y=1.7cm]
	   \tikzset{vertex/.style = {shape=circle,draw,minimum size=10.5em}}
	    \tikzset{edge/.style = {-,> = latex'}}
		\node (a) at (1,2) {$32123$};
		\node (b) at (1,1) {$31213$};
		\node (c) at (2,0) {$31231$};
		\node (d) at (1,-1) {$13231$};
		\node (e) at (1,-2) {$12321$};
		\node (f) at (0,0) {$13213$};
		\draw (a)--(b)--(c)--(d)--(e);
		\draw (d)--(f)--(b);
		\end{tikzpicture}
	\end{center}
	\caption{The graph $\mathcal{G}_{4231}$ of the reduced words of $\sigma=4231$}
\end{figure}

\begin{lemma}
Let $\sigma\in\mathcal{Z}$ such that $\sigma\in\frak{S}_{r}$. Then $\ell(\sigma)=r+1$.
\end{lemma}
\begin{proof}
Since $\sigma\in\frak{S}_{r}$ is of the form  $2,3,\cdots,r-3, r, r-2 , r-1,1$, there are exactly 5 inversions in $r,r-2,r-1, 1$ and $r-4$ inversions in $2,3,\dots,r-3$. Hence, $\ell(\sigma) = r+1$.      
\end{proof} 
\noindent It turns out that the number of reduced words for any given permutation $\sigma\in\mathcal{Z}$ is expressible in terms of the length of $\sigma$.
\begin{proposition}
For any permutation $\sigma\in\mathcal{Z}$, the number of reduced words for $\sigma$ is precisely ${\ell(\sigma)-1\choose 2}$.
\end{proposition}
\begin{proof}
Notice that the set $R(\sigma)$ of reduced words of any permutation $\sigma\in \mathcal{Z}$  such that $\sigma\in\frak{S}_{r}$ can be broken into two disjoint subsets A and B. The set A consists of reduced words whose last number is $r-1$ while set B contains those whose last number is $i$ such that $i=r-3$.  $|A|={{r-1}\choose 2}$ and $|B|= r-1$. Therefore, $|R(\sigma)|={\ell(\sigma)-1\choose 2}$. 
\end{proof}
\begin{remark}
The implication of the proposition above is  the fact that  reduced words of permutations in $\mathcal{Z}$ are counted by triangular numbers thereby establishing a connection between the graph $\mathcal{G}_{\sigma}$ of reduced words of  the permutation $\sigma\in\mathcal{Z}$ and the Ferrers diagram of the staircase partition $\lambda\vdash|R(\sigma)|$ of the cardinality of the reduced words of $\sigma$. This is carefully stated in what follows.
\end{remark}
\begin{proposition}
For $\ell(\lambda)\geq 3$, there is a bijection between the set $\mathcal{P}_{\lambda}$ of staircase partitions $\lambda$ and the family $\mathcal{B}$ of graphs $\mathcal{G}_{\sigma}$ of reduced words of $\sigma\in \mathcal{Z}$.
\end{proposition}
\begin{proof}
 Suppose that $\sigma\in\mathcal{Z}$,  there exists a unique  symmetric group $\frak{S}_{r}, r\geq 4$ containing $\sigma$.  So the graph $\mathcal{G}_{\sigma}$ of reduced words of $\sigma$ can be identified with the staircase partition $\lambda\in\mathcal{P}_{\lambda}$ given  by $\lambda\vdash|R(\sigma)|$ such that  $r=\ell(\lambda)+1$.   The vertices of $\mathcal{G}_{\sigma}$ correspond to the boxes of the Ferrers diagram of shape $\lambda$. 
\end{proof}
\begin{example} Identification of the graph $\mathcal{G}_{236451}$ with  the Ferrers diagram of the staircase partition $\lambda = (5,4,3,2,1)$ is shown in Figure 4.
\begin{figure}[!hbt]
\begin{center}
\begin{tikzpicture}[scale=0.47]
			\draw (0,0)--(5,0);
			\draw (0,1)--(5,1);
			\draw (5,1)--(5,0);
			\draw (0,0)--(0,5);
			\draw (1,0)--(1,5);
			\draw (0,5)--(1,5);
			 \draw (0,2)--(4,2);
			 \draw (4,2)--(4,0);
			 \draw (0,3)--(3,3)--(3,0);
			 \draw (0,4)--(2,4)--(2,0);
                          \draw[ultra thin,red](0.5, 0.5)--(1.5,0.5)--(2.5,0.5)--(3.5,0.5)--(4.5,0.5);
                          \draw[ultra thin,red](0.5,0.5)--(0.5, 1.5)--(1.5,1.5)--(2.5,1.5)--(3.5,1.5);
                          \draw[ultra thin,red](0.5, 1.5)--(0.5,2.5)--(1.5,2.5)--(2.5,2.5);
                          \draw[ultra thin,red](0.5, 2.5)--(0.5,3.5)--(1.5,3.5);
                          \draw[ultra thin,red](0.5,3.5)--(0.5,4.5);
                           \draw[ultra thin,red](1.5,0.5)--(1.5,1.5)--(1.5,2.5)--(1.5,3.5);
                           \draw[ultra thin,red](2.5,0.5)--(2.5,1.5)--(2.5,2.5);
                           \draw[ultra thin,red](3.5,0.5)--(3.5,1.5);
                           \draw[fill] (0.5,0.5) circle [radius=0.075];
                           \draw[fill] (1.5,0.5) circle [radius=0.075];
                           \draw[fill] (2.5,0.5) circle [radius=0.075];
                           \draw[fill] (3.5,0.5) circle [radius=0.075];
                           \draw[fill] (4.5,0.5) circle [radius=0.075];
                           \draw[fill] (0.5,1.5) circle [radius=0.075];
                           \draw[fill] (1.5,1.5) circle [radius=0.075];
                           \draw[fill] (2.5,1.5) circle [radius=0.075];
                           \draw[fill] (3.5,1.5) circle [radius=0.075];
                           \draw[fill] (0.5,2.5) circle [radius=0.075];
                           \draw[fill] (1.5,2.5) circle [radius=0.075];
                           \draw[fill] (2.5,2.5) circle [radius=0.075];
                           \draw[fill] (0.5,3.5) circle [radius=0.075];
                           \draw[fill] (1.5,3.5) circle [radius=0.075];
                           \draw[fill] (0.5,4.5) circle [radius=0.075];
			\end{tikzpicture}
\end{center}
\caption{ $\mathcal{G}_{236451}\mapsto\lambda=(5,4,3,2,1)$}
\end{figure}
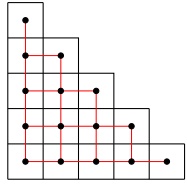 
\end{example}

\noindent The bijection allows us to view  every graph $\mathcal{G}_{\sigma}\in\mathcal{B}$ of  reduced words  of the permutations of $\mathcal{Z}$  as a pseudo multipartite simple graph $B_{\lambda}$ where $\lambda$ is the staircase partition such that $\lambda\vdash|R(\sigma)|$.  The description of $B_{\lambda}$ is  as follows:
\noindent The vertex set $V$ of $B_{\lambda}$ is a disjoint union of ordered subsets $V_{1},V_{2},\dots, V_{\ell(\lambda)}$ such that
\begin{equation}
|V_{i}|=\ell(\lambda)+1-i \ {\rm and} \ 1\leq i\leq \ell(\lambda).
\end{equation}
For example, the pseudo multipartite simple graph $B_{6,5,4,3,2,1}$  representing the reduced word graph $\mathcal{G}_{2347561}$ is illustrated below. 

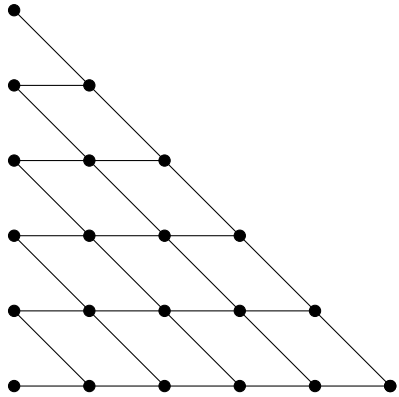
\begin{figure}[!hbt]
\begin{center}
\begin{tikzpicture}
\draw  (0,0)--(1,0)--(2,0)--(3,0)--(4,0)--(5,0);
\draw (0,5)--(1,4)--(2,3)--(3,2)--(4,1)--(5,0);
\draw (1,4)--(0,4)--(1,3)--(2,2)--(3,1)--(4,0);
\draw(2,3)--(1,3)--(0,3)--(1,2)--(2,1)--(3,0);
\draw(3,2)--(2,2)--(1,2)--(0,2)--(1,1)--(2,0);
\draw(4,1)--(3,1)--(1,1)--(0,1)--(1,0);
\draw[fill] (1,3) circle [radius=0.075];
\draw[fill] (2,2) circle [radius=0.075];
\draw[fill] (3,1) circle [radius=0.075];
\draw[fill] (4,0) circle [radius=0.075];
\draw[fill] (0,4) circle [radius=0.075];
\draw[fill] (1,0) circle [radius=0.075];
\draw[fill] (0,0) circle [radius=0.075];
\draw[fill] (3,0) circle [radius=0.075];
\draw[fill] (2,0) circle [radius=0.075];
\draw[fill] (1,1) circle [radius=0.075];
\draw[fill] (0,3) circle [radius=0.075];
\draw[fill] (2,1) circle [radius=0.075];
\draw[fill] (1,2) circle [radius=0.075];
\draw[fill] (0,2) circle [radius=0.075];
\draw[fill] (0,1) circle [radius=0.075];
\draw[fill] (5,0) circle [radius=0.075];
\draw[fill] (0,5) circle [radius=0.075];
\draw[fill] (1,4) circle [radius=0.075];
\draw[fill] (2,3) circle [radius=0.075];
\draw[fill] (3,2) circle [radius=0.075];
\draw[fill] (4,1) circle [radius=0.075];
\draw[fill] (5,0) circle [radius=0.075];
\end{tikzpicture}
\end{center}
\caption{$B_{6,5,4,3,2,1}$ }
\end{figure}
\noindent Notice that the vertices in each disjoint subset $V_k$ are disconnected, as a result there is no path among them. The observation gives rise to the edge-missing polynomial $P_{\ell(\lambda)}(e)$ identified with the pseudo-multipartite graph $B_{\lambda}$ given by
\begin{equation}
P_{\ell(\lambda)}(e)=\sum^{\ell(\lambda)}_{k=1} (\ell(\lambda)+1-k)e^{\ell(\lambda)-k}.
\end{equation}
 It is interpreted as follows:  $\ell(\lambda)+1-k$ vertices belonging to the vertex subset $V_{k}$ lack $\ell(\lambda)-k$ edges to form a simple path. The edge-missing polynomial of $B_{6,5,4,3,2,1}$ is given by  $P_{6}(e) =6e^{5}+5e^{4}+4e^{3}+3e^{2}+2e+1$. The generating function of the family $\mathcal{B}$ of pseudo-multipartite graphs is given by
 \begin{equation}
 \sum_{\ell(\lambda)\geq 2} P_{\ell(\lambda)}(e) z^{\ell(\lambda)} = \frac{z}{(1-e)(1-ez)^2}. 
 \end{equation} 
Our interest is in the  bicolouring of the generating function. This is closely related to the checkerboard colouring of Ferrers diagrams of staircase partitions which stems from the action of the abelian group $G$ diagonally embedded in ${\rm SL}_{2}({\C})$ studied in [4] The visualisation of colouring here is slightly different in the sense that the bicolouring of the graph $B_{\lambda}$ can be identified with the diagonal colouring of the Ferrers diagram of shape$(\lambda)$ which comes from a $1$-dimensional G-eigenspace in the representation $\rho_{a+b \ {\rm mod2}}$ spanned by the monomial $x^{a}y^{b}\in\C(\N ^{2})$ under the canonical extension of the  action of $G$ to the coordinate ring $\C[x,y]$.
\begin{theorem}
Let  $B_{\lambda}$ be the pseudo multipartite simple graph associated with the graph $\mathcal{G}_{\sigma}$ of reduced words of $\sigma\in\frak{S}_{r}, r\geq 4$ such that  $r=\ell(\lambda)+1$. Then  for $\ell(\lambda)\geq 3$,  $B_{\lambda}$ has  ${{\ell(\lambda)+1}\choose 2}$ vertices , $\ell(\lambda)(\ell(\lambda)+1)$ edges, $\ell(\lambda)-1$ of which are given by braid relation, and  ${\ell(\lambda)-1\choose 2}$ 4- cycles.
\end{theorem}
\begin{proof}
Since vertices are reduced words of $\sigma$ and $\ell(\sigma)=\ell(\lambda)+2$, so the number follows. Notice that $B_{\lambda}$ is finite and simple and that the edges are given by the finite sequence $2, 4,\dots, 2(\ell(\lambda)-1)$, the terms sum up to $\ell(\lambda)(\ell(\lambda)+1)$. Likewise, the cycles are given by the finite sequence $1, 2, 3,\dots,\ell(\lambda)-2$  which sums up to ${\ell(\lambda)-1\choose 2}$ 4-cycles. 
\end{proof}
\begin{corollary}
Given any pseudo multipartite simple graph  $B_{\lambda}\in\mathcal{B}$ such that it has $v$ vertices, $e$ edges and $c$ cycles,  then  $v+c-e=1$.
\end{corollary}

\noindent Notice that the ordered disjoint subsets $V_{1},V_{2},\dots, V_{\ell(\lambda)}$ of the vertex set $V$ of the graph $B_{\lambda}$ correspond to the sets $D_{(a,b)}$ of diagonal boxes of the Ferrers diagram of the staircase partition $\lambda$ such that $V_{i}$ is identified with $D_{(a,b)}$ when $a+b = i-1,  1\leq i\leq \ell(\lambda)$. Two vertices in $B_{\lambda}$ share an edge  if and only if two boxes in the Ferrers diagram share a side. It turns out that the members of the family $\mathcal{B}$ of pseudo multipartite graphs can be ordered in line with the bijection described above. In fact, the well-orderedness  of the sequence of triangular numbers induces a natural order given by inclusion on the members of $\mathcal{B}$.\\

\begin{proposition}
Let $\lambda^{1}$ and $\lambda^{2}$ be the staircase partitions  associated to $B_{\lambda^{1}}$ and $B_{\lambda^{2}}$ respectively in $\mathcal{B}$. Then $B_{\lambda^{1}}$ is a subgraph of $ B_{\lambda^{2}}$  if and only if  ${\rm shape}(\lambda^{1})\subset {\rm shape}(\lambda^{2}).$ 
\end{proposition}
\begin{proof}
Suppose that ${\rm shape}(\lambda^{1})\subset {\rm shape}(\lambda^{2})$, this implies that $\ell(\lambda^{1})<\ell(\lambda^{2})$, so the collection of disjoint ordered subsets $V_{1}, V_{2},\dots, V_{\ell(\lambda^{1})}$  of the vertex set $V^{1}$ belonging to the pseudo-multipartite simple graph $B_{\lambda^{1}}$ is a proper sub-collection of the disjoint ordered subsets of the vertex set $V^{2}$ belonging to  $B_{\lambda^{2}}$, since every edge in $B_{\lambda^{1}}$ has both vertices in $V^{1}$, therefore,
$B_{\lambda^{1}}$ is a subgraph of $ B_{\lambda^{2}}$. Conversely, suppose that $B_{\lambda^{1}}$ is a subgraph of $ B_{\lambda^{2}}$, this implies that the collection of  disjoint ordered subsets of the vertex set $V^{2}$ belonging to  $B_{\lambda^{2}}$ contains the disjoint ordered subsets of the vertex set $V^{1}$ belonging to  $B_{\lambda^{1}}$, hence $\ell(\lambda^{1})<\ell(\lambda^{2})$, therefore, ${\rm shape}(\lambda^{1})\subset {\rm shape}(\lambda^{2}).$
\end{proof}
\begin{defn} 
We say two multipartite graphs  $B_{\lambda^{1}}$ and  $B_{\lambda^{2}}$ in the ordered family $\mathcal{B}$ are consecutive if  $B_{\lambda^{1}}$ is a subgraph of $B_{\lambda^{2}}$ and there does not exist a staircase partition $\lambda^{3}$ such that  $B_{\lambda^{3}}$ is a subgraph of $B_{\lambda^{2}}$ which contains $B_{\lambda^{1}}$ as a subgraph.
\end{defn}

\begin{defn}
The multipartite graph $B_{\lambda^{1}}$ is said to be situated in the odd ( resp. even) position if the length $\ell( \lambda^{1})$ of $\lambda$  is odd (resp.even). 
\end{defn}

\begin{defn}
A consecutive pair $(B_{\lambda^{1}}, B_{\lambda^{2}})$  is said to be in parity if  both $B_{\lambda^{1}}$ and $B_{\lambda^{2}}$ have even  or odd number of vertices. 
\end{defn}

\noindent For instance,  the pair $(B_{3,2,1}, B_{4,3,2,1})$ is parity  while $(B_{4,3,2,1}, B_{5,4,3,2,1})$ is not. Recall that staircase partitions are self-conjugates and the set of self-conjugate partitions of $n$ is in bijection with the set of distinct odd parts partitions of $n$. This is very important in what follows.
\begin{theorem}
Let $B_{\lambda^{1}}$ and  $B_{\lambda^{2}}$  be a consecutive pair in the  family $\mathcal{B}$. Then the following statements are equivalent
\begin{enumerate}
\item[(i)] $B_{\lambda^{1}}$ and  $B_{\lambda^{2}}$ are  parity pair 
\item[(ii)] the length $\ell(\lambda^{1})$ is odd.
\item[(iii)] the distinct odd parts partitions $\mu^{1}$ and $\mu^{2}$ corresponding to $\lambda^{1}$ and $\lambda^{2}$ respectively share the same length.
\end{enumerate}
\end{theorem}
\begin{proof}
$(i)\Rightarrow(ii)$ Suppose that $B_{\lambda^{1}}$ and  $B_{\lambda^{2}}$ are  parity pair, we recall that  addition of an even number to any number preserves the parity of the number. Since both $|\lambda^{1}|$ and $|\lambda^{2}|$ are either odd or even, and $|\lambda^{2}|=\ell(\lambda^{1})+1+|\lambda^{1}|$, it implies $\ell(\lambda^{1})+1$ must be even, so $\ell(\lambda^{1})$ is odd.\\
$(ii)\Rightarrow(iii)$ Since $\lambda^{1}$ is a self conjugate partition, denote by $\mu^{1}$ the corresponding distinct odd parts partition. These distinct parts are given by the set
\begin{equation}
\mathcal{C}_{\lambda^{1}} = \left\lbrace H_{i,i} \subset {\rm shape}(\lambda^{1}): i\in\left[\left\lceil \frac{\ell(\lambda^{1})}{2}\right\rceil \right]\right\rbrace
\end{equation}
where $H_{i,i}$ is the hook  of the $(i,i)$ box along the north-east diagonal of the shape$(\lambda^{1})$, so the length $\ell(\mu^{1})$ of the partition $\mu^{1}$ is the cardinality $|\mathcal{C}_{\lambda^{1}}|$ of the set $\mathcal{C}_{\lambda^{1}}$. Since $|\lambda^{2}|=\ell(\lambda^{1})+1+|\lambda^{1}|$ and $\ell(\lambda^{1})$ is odd, hence, the shape$(\lambda^{2})$ amounts to adding  $\ell(\lambda^{1})+1$ boxes to the shape$(\lambda^{1})$, each to every corner, since $\lambda^{2}$ is also a staircase partition. Denote by $\mu^{2}$ the distinct odd parts partition corresponding to $\lambda^{2}$.  The length $\ell(\mu^{2})=|\mathcal{C}_{\lambda^{2}}|$ coincides with $|\mathcal{C}_{\lambda^{1}}|$
 since $\ell(\lambda^{1})+1$ is even.\\
$(iii)\Rightarrow (i)$ Since the distinct odd parts partitions $\mu^{1}$ and $\mu^{2}$ corresponding to $\lambda^{1}$ and $\lambda^{2}$ respectively share the same length, hence $\lambda^{2}|=\ell(\lambda^{1})+1+|\lambda^{1}|$ and $\ell(\lambda^{1})+1$ is even.  Therefore, parity is preserved, so $B_{\lambda^{1}}$ and  $B_{\lambda^{2}}$ are  parity pair.
\end{proof}
\begin{corollary}
Let $B_{\lambda^{1}}$ and  $B_{\lambda^{2}}$ be a parity pair. Then both $B_{\lambda^{1}}$ and  $B_{\lambda^{2}}$ have even (resp. odd)  number  of vertices  if and only if  $\ell(\lambda^{1})\equiv1\mod 4$ (resp. $\ell(\lambda^{1})\equiv 3\mod4$).
\end{corollary}

\section{Colouring of the family of multipartite graphs $B_{\lambda}$ }

\noindent We now begin the discussion of a proper colouring $\alpha$ of the vertices of every pseudo multipartite simple graph $B_{\lambda}\in\mathcal{B}$.  Given a set  $\{1,2,3,\dots,\}$ of colours, by a proper colouring $\alpha$ we mean a function 
$$\alpha: V_{B_{\lambda}}\longrightarrow \{1,2,3,\dots\}$$
such that  if $v_1, v_2\in V_{B_{\lambda}}$ are adjacent, then $\alpha(v_1)\ne \alpha(v_2)$. The goal here is to characterise all chromatic polynomials belonging to the members of  $\mathcal{B}$.  Recall that every $B_{\lambda}$ is not cycle free. In fact, according to the Theorem 2.7, $B_{\lambda}$ admits only 4-cycles.  The following Lemma is key to our discussion.
\begin{lemma}
Let $C_{4}^{2d+2}, \ d\geq 1$ be a $d$ 4-cycle graph with $2d+2$ verices. Then its chromatic polynomial is given by 
$$\chi_{C_{4}^{2d+2}}(k)= k(k-1)(k^2-3k+3)^{d}.$$
\end{lemma}
\begin{proof}
Using induction, first, consider the one 4-cycle and two 4-cycle graphs illustrated below.  Notice that the chromatic polynomial for the one 4-cycle graph  $C_{4}$ is simple, it is given by $\chi_{C_{4}}=k(k-1)+k(k-1)(k-2)^2=k(k-1)(k^2-3k+3)$ and the chromatic polynomial of the two 4-cycle graph $C_{6}$ is given by $\chi_{C_6}(k)=k(k-1)(k^4-6k^3+15k^2-18k+9)= (k^2-3k+3)\cdot \chi_{C_4}(k)$. Therefore, the chromatic polynomial of the $d$ 4-cycle graph $C_{4}^{2d+2}$ is precisely given by $\chi_{C_{4}^{2d+2}}(k)=(k^2-3k+3)^{d-1}\cdot\chi_{C_{4}}(k)$.
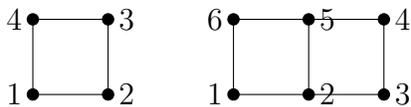
\begin{center}
 \begin{tikzpicture}
\draw (0,1)--(0,0)--(1,0)--(1,1)--(0,1);
\draw[fill] (0,0) circle [radius=0.075];
\draw[fill] (0,1) circle [radius=0.075];
\draw[fill] (1,1) circle [radius=0.075];
\draw[fill] (1,0) circle [radius=0.075];
\node[left] at (0,0){1};
\node[right] at (1,0){2};
\node[left] at (0,1){4};
\node[right] at (1,1){3};
\end{tikzpicture}  \ \ \ \
\begin{tikzpicture}
\draw (0,1)--(0,0)--(1,0)--(1,1)--(0,1);
\draw (1,1)--(2,1)--(2,0)--(1,0);
\draw[fill] (0,0) circle [radius=0.075];
\draw[fill] (0,1) circle [radius=0.075];
\draw[fill] (1,1) circle [radius=0.075];
\draw[fill] (1,0) circle [radius=0.075];
\draw[fill] (2,0) circle [radius=0.075];
\draw[fill] (2,1) circle [radius=0.075];
\node[left] at (0,0){1};
\node[right] at (1,0){2};
\node[left] at (0,1){6};
\node[right] at (1,1){5};
\node[right] at (2,1){4};
\node[right] at (2,0){3};
\end{tikzpicture} 
\end{center}
\end{proof}

\begin{theorem}
Let $B_{\lambda}$ be the multipartite graph indexed by the staircase partition $\lambda$. Then for $\ell(\lambda)\geq 3$,  its chromatic polynomial $\chi_{B_{\lambda}}(k)$ is given by
$$\chi_{B_{\lambda}}(k)= k(k-1)^{3}(k^2-3k+3)^{m} \  {\rm where } \ m= {{\ell(\lambda)-1}\choose 2}.$$
\end{theorem}
\begin{proof}
We shall use induction for the proof.  Notice that $B_{\lambda}$ contains ${{\ell(\lambda)-1}\choose 2}$ 4-cycles such that $\ell(\lambda)\geq 3$. Now Consider  the pseudo-multipartite graph $B_{3,2,1}$ with only one 4-cycle. Using the technique of deletion-contraction recursion we have

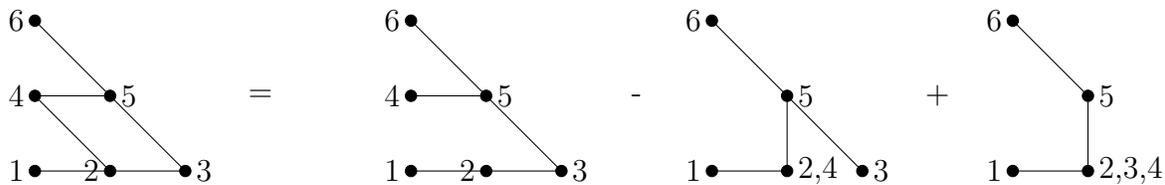
\begin{tikzpicture}
\draw(0,0)--(1,0)--(2,0);
\draw (0,1)--(1,1);  
\draw (0,1)--(1,0);
\draw(0,2)--(1,1)--(2,0);
\draw[fill] (0,0) circle [radius=0.075];
\draw[fill] (0,1) circle [radius=0.075];
\draw[fill] (0,2) circle [radius=0.075];
\draw[fill] (1,1) circle [radius=0.075];
\draw[fill] (1,0) circle [radius=0.075];
\draw[fill] (2,0) circle [radius=0.075];
\node[left] at (0,0){1};
\node[left] at (1,0){2};
\node[right] at (2,0){3};
\node[left] at (0,1){4};
\node[right] at (1,1){5};
\node[left] at (0,2){6};

\node at (3,1){=};

\draw(5,0)--(6,0)--(7,0);
\draw (5,1)--(6,1);
\draw(5,2)--(6,1)--(7,0);
\draw[fill] (5,0) circle [radius=0.075];
\draw[fill] (5,1) circle [radius=0.075];
\draw[fill] (5,2) circle [radius=0.075];
\draw[fill] (6,1) circle [radius=0.075];
\draw[fill] (6,0) circle [radius=0.075];
\draw[fill] (7,0) circle [radius=0.075];
\node[left] at (5,0){1};
\node[left] at (6,0){2};
\node[right] at (7,0){3};
\node[left] at (5,1){4};
\node[right] at (6,1){5};
\node[left] at (5,2){6};

\node at (8,1){-};

\draw(9,0)--(10,0);
\draw (10,0)--(10,1);
\draw(9,2)--(10,1)--(11,0);
\draw[fill] (9,0) circle [radius=0.075];
\draw[fill] (9,2) circle [radius=0.075];
\draw[fill] (10,1) circle [radius=0.075];
\draw[fill] (10,0) circle [radius=0.075];
\draw[fill] (11,0) circle [radius=0.075];
\node[left] at (9,0){1};
\node[right] at (10,0){2,4};
\node[right] at (11,0){3};
\node[right] at (10,1){5};
\node[left] at (9,2){6};

\node at (12,1){+};

\draw(13,0)--(14,0);
\draw (14,0)--(14,1);
\draw(13,2)--(14,1);
\draw[fill] (13,0) circle [radius=0.075];
\draw[fill] (13,2) circle [radius=0.075];
\draw[fill] (14,1) circle [radius=0.075];
\draw[fill] (14,0) circle [radius=0.075];
\node[left] at (13,0){1};
\node[right] at (14,0){2,3,4};
\node[right] at (14,1){5};
\node[left] at (13,2){6};
\end{tikzpicture}

\ \\
So, $\chi_{B_{3,2,1}}(k)=k(k-1)^{5}-k(k-1)^{4}+k(k-1)^{3}=k(k-1)^{3}(k^2-3k+3)$. Notice that $\chi_{B_{3,2,1}}(k)=(k-1)^{2}\cdot\chi_{C_{4}}(k)$ since $B_{3,2,1}$ contains only one 4-cycles and there are $k-1$ colourings for each of the vertices $1$ and $6$. Following Lemma 4.3, the chromatic polynomial of the graph $B_{\lambda}$ consisting of ${{\ell(\lambda)-1}\choose 2}$ $4$-cycles is precisely given by $\chi_{B_{\lambda}}(k)=(k-1)^{2}\cdot\chi_{C_{4}^{2m+2}}(k)= k(k-1)^{3}(k^2-3k+3)^{m}$, where $m={{\ell(\lambda)-1}\choose 2}$.
\end{proof}
\begin{corollary}
For $\ell(\lambda)\geq 3$,  the chromatic numberl $\chi(B_{\lambda})$ is 2 .
\end{corollary}
\noindent The fact that $B_{\lambda}$ is bicoloured allows us to split the staircase partition $\lambda$ into two. That is, $\lambda =\mu\cup\kappa$ such that
\begin{equation}
|\mu|=\sum_{r=0}^{\lfloor \frac{\ell(\lambda}{2}\rfloor} |V_{2r+1}| \  {\rm and}  \  |\kappa|=\sum_{r=1}^{\lfloor \frac{\ell(\lambda}{2}\rfloor}|V_{2r}|. 
\end{equation}
when $\ell(\lambda)$ is odd and 
\begin{equation}
|\mu|=\sum_{r=1}^{\lfloor \frac{\ell(\lambda}{2}\rfloor} |V_{2r-1}| \  {\rm and}  \  |\kappa|=\sum_{r=1}^{\lfloor \frac{\ell(\lambda}{2}\rfloor}|V_{2r}|. 
\end{equation}
when $\ell(\lambda)$ is even. The size $|\mu| \ ({\rm resp.} |\kappa|)$ is the total number of vertices in the union of sub-collection of the ordered disjoint vertex subsets $V_1, V_2,\dots, V_{\ell(\lambda)}$ indexed by odd numbers (resp. even numbers). In other words, when $\ell(\lambda)$ is odd, $|\mu| \ ({\rm resp.} |\kappa|)$ is the total number vertices  painted in back (reps. red) and it is otherwise when $\ell(\lambda)$ is even.  We denote this bi-colouring by $\mathcal{C}^{B_{\lambda}}_{|\mu|, |\kappa|}$  and call it the 2-colour separation of the vertices  of the pseudo-multipartite graph $B_{\lambda}$. For instance, the 2-colour separation $\mathcal{C}^{B_{5,4,3,2,1}}_{9, 6}$ of the pseudo multipartite graph $B_{5,4,3,2,1}$ is illustrated in Figure 6. This corresponds to the diagonal colouring of the associated Ferrers diagram. For any given 2-colour separation $\mathcal{C}^{B_{\lambda}}_{|\mu|, |\kappa|}$, we associate with it the difference $k:=|\mu|-|\kappa|$ and call it the balance of the separation. Notice that $k$ is strictly positive.\\
\begin{proposition}
Let  $\mathcal{C}^{B_{\lambda^{1}}}_{|\mu|, |\kappa|}$  be the $2$-colour separation of the vertices of the multipartite graph $B_{\lambda}$. Then the balance of  the separation satisfies $k\leq\lceil \frac{\ell(\lambda)}{2}\rceil$. The bound is very sharp.
\end{proposition}
\begin{proof}
Suppose that $\ell(\lambda)$ is odd, we have $|\mu|= \lceil \frac{\ell(\lambda)}{2}\rceil^{2}$ and $|\kappa| = \lfloor \frac{\ell(\lambda)}{2}\rfloor(\lfloor \frac{\ell(\lambda)}{2}\rfloor+1)$. Notice that  $\lceil \frac{\ell(\lambda)}{2}\rceil^{2}-\lfloor \frac{\ell(\lambda)}{2}\rfloor^{2} = \lceil \frac{\ell(\lambda)}{2}\rceil + \lfloor \frac{\ell(\lambda)}{2}\rfloor$, so $k\leq \lceil \frac{\ell(\lambda)}{2}\rceil$, since $\lfloor \frac{\ell(\lambda)}{2}\rfloor< \lceil \frac{\ell(\lambda)}{2}\rceil$. On the other hand, suppose that $\ell(\lambda)$ is even, this implies that $\lfloor \frac{\ell(\lambda)}{2}\rfloor = \lceil \frac{\ell(\lambda)}{2}\rceil$, so $k = \lceil \frac{\ell(\lambda)}{2}\rceil$, therefore the bound is very sharp.
\end{proof}
\begin{proposition}
Let $B_{\lambda^{1}}$ and  $B_{\lambda^{2}}$  be a parity pair. Then their respective 2-colour separations $\mathcal{C}^{B_{\lambda^{1}}}_{|\mu^{1}|, |\kappa^{1}|}$ and $\mathcal{C}^{B_{\lambda^{2}}}_{|\mu^{2}|, |\kappa^{2}|}$ share the same balance.
\end{proposition}
\begin{proof}
By definition, it is either both  $|{\lambda^{1}}|$ and  $|{\lambda^{2}}|$ are even, this implies that    $|{\lambda^{1}}|$ and  $|{\lambda^{2}}|$ are respectively in the odd and even positions of the sequence  of triangular numbers according to the Theorem 2.13 , therefore, there exists $k\geq 1$ such that $|{\lambda^{1}}| :=2k^{2}-k$ and $|{\lambda^{2}}|:=2k^{2}+k$ are respectively the $i^{th}$ and $(i+1)^{th}$ terms of the sequence for odd $i$. Hence, $|\lambda^{1}|=k^2$, $|\mu^{1}|=k^2-k$ and $|\lambda^{2}|=k^2+k$, $|\mu^{2}|=k^2$, so the  2-colour separations of $B_{\lambda^{1}}$ and $B_{\lambda^{2}}$ share balance $k$.
\end{proof}
\begin{remark}
The shared balance $k$  of the colour separations $\mathcal{C}^{B_{\lambda^{1}}}_{|\mu^{1}|, |\kappa^{1}|}$ and $\mathcal{C}^{B_{\lambda^{2}}}_{|\mu^{2}|, |\kappa^{2}|}$ identified with the pseudo multipartite partitions $B_{\lambda^{1}}$ and  $B_{\lambda^{2}}$ respectively is the number of  hooks in the respective Ferrers diagrams of shapes $\lambda^{1}$ and $\lambda^{2}$, see {\it Example 1.1}. In other words, $k$ is precisely the length  shared by  the distinct odd parts partitions $\eta^{1}$ and $\eta^{2}$ corresponding to the staircase partitions  $\lambda^{1}$ and $\lambda^{2}$ respectively.
 \end{remark}
 
\begin{figure}[htbp]
\begin{center} 
\begin{tikzpicture}
\draw (0,0)--(1,0)--(2,0)--(3,0)--(4,0);
\draw (0,4)--(1,3)--(2,2)--(3,1)--(4,0);
\draw(1,3)--(0,3)--(1,2)--(2,1)--(3,0);
\draw(2,2)--(1,2)--(0,2)--(1,1)--(2,0);
\draw(3,1)--(1,1)--(0,1)--(1,0);
\draw[fill=red] (1,3) circle [radius=0.075];
\draw[fill] (2,2) circle [radius=0.075];
\draw[fill=red] (3,1) circle [radius=0.075];
\draw[fill] (4,0) circle [radius=0.075];
\draw[fill] (0,4) circle [radius=0.075];
\draw[fill=red] (1,0) circle [radius=0.075];
\draw[fill] (0,0) circle [radius=0.075];
\draw[fill=red] (3,0) circle [radius=0.075];
\draw[fill] (2,0) circle [radius=0.075];
\draw[fill=red] (1,1) circle [radius=0.075];
\draw[fill] (0,3) circle [radius=0.075];
\draw[fill] (2,1) circle [radius=0.075];
\draw[fill=red] (1,2) circle [radius=0.075];
\draw[fill] (0,2) circle [radius=0.075];
\draw[fill] (0,1) circle [radius=0.075];
\node at (2,-0.5){$\mathcal{C}^{B_{5,4,3,2,1}}_{9, 6}$ };

\draw (6,0)--(11,0);
\draw (6,1)--(11,1);
\draw (11,1)--(11,0);
\draw (6,0)--(6,5);
\draw (6,5)--(7,5);
\draw (7,5)--(7,0);
\draw (6,2)--(10,2)--(10,0);
\draw (6,3)--(9,3)--(9,0);
\draw(6,4)--(8,4)--(8,0);
\node at (9,-0.5){{\rm Colouring \ of } \ $\lambda=(5,4,3,2,1)$ };

\filldraw[fill=black!90!white, draw=black, opacity=0.8] (6,0)--(7,0)--(7,1)--(6,1)--(6,0);
\filldraw[fill=black!90!white, draw=black, opacity=0.8] (6,3)--(7,3)--(7,2)--(6,2)--(6,3);
\filldraw[fill=black!90!white, draw=black, opacity=0.8] (6,5)--(7,5)--(7,4)--(6,4)--(6,5);
\filldraw[fill=red!90!white, draw=black, opacity=0.8] (6,2)--(7,2)--(7,1)--(6,1)--(6,2);
\filldraw[fill=red!90!white, draw=black, opacity=0.8] (6,4)--(7,4)--(7,3)--(6,3)--(6,4);
\filldraw[fill=black!90!white, draw=black, opacity=0.8] (7,4)--(8,4)--(8,3)--(7,3)--(7,4);
\filldraw[fill=red!90!white, draw=black, opacity=0.8] (7,3)--(8,3)--(8,2)--(7,2)--(7,3);
\filldraw[fill=black!90!white, draw=black, opacity=0.8] (7,2)--(8,2)--(8,1)--(7,1)--(7,2);
\filldraw[fill=red!90!white, draw=black, opacity=0.8] (7,1)--(8,1)--(8,0)--(7,0)--(7,1);
\filldraw[fill=black!90!white, draw=black, opacity=0.8] (8,3)--(9,3)--(9,2)--(8,2)--(8,3);
\filldraw[fill=red!90!white, draw=black, opacity=0.8] (8,2)--(9,2)--(9,1)--(8,1)--(8,2);
\filldraw[fill=black!90!white, draw=black, opacity=0.8] (8,1)--(9,1)--(9,0)--(8,0)--(8,1);
\filldraw[fill=black!90!white, draw=black, opacity=0.8] (9,2)--(10,2)--(10,1)--(9,1)--(9,2);
\filldraw[fill=red!90!white, draw=black, opacity=0.8] (9,1)--(10,1)--(10,0)--(9,0)--(9,1);
\filldraw[fill=black!90!white, draw=black, opacity=0.8] (10,1)--(11,1)--(11,0)--(10,0)--(10,1);

\end{tikzpicture}
\end{center}
\caption{The 2-colour separation $\mathcal{C}^{B_{5,4,3,2,1}}_{9, 6}$ and the corresponding coloured Ferrers diagram. }
\end{figure}
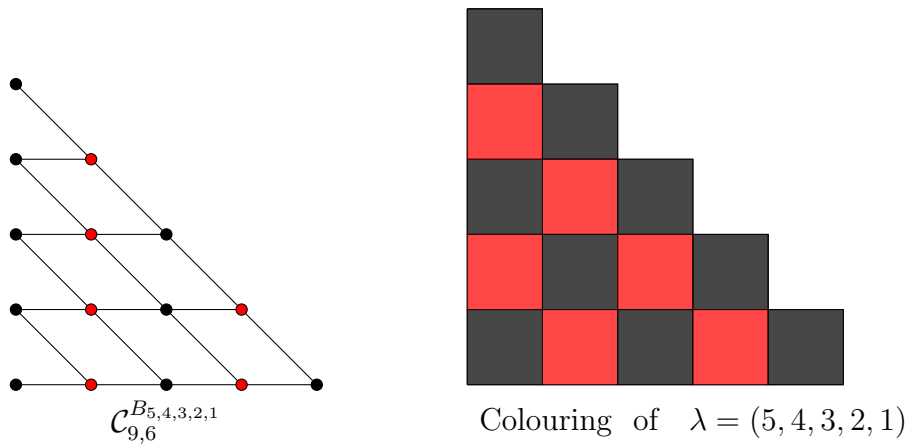
\noindent The collection $\mathcal{V}$ of parity pairs in $\mathcal{B}$ deepens the understanding of the colouring of multipartite graphs $B_{\lambda}$. This shall be discussed in detail in section 5. In fact,  $\mathcal{V}$ can be identified with a certain subcollection of $2\times 2$ nonsingular matrices. \\
\ \\
\begin{theorem}
There is a correspondence between the set 
$$\mathcal{V}=\{(B_{\lambda^{1}}, B_{\lambda^{2}})\in \mathcal{B}\times\mathcal{B} :  \ell(\lambda^{1}) \ {\rm  is \  odd}\}$$ and the set $\mathcal{Y}$ of $2\times 2$ invertible matrices with determinant  $k^{2}$ given by 
$$\mathcal{Y} = \left\{\begin{bmatrix}  k^{2} &  k^{2}+ k \\   k^{2}- k &  k^{2} \end{bmatrix}\in {\rm GL}_{2}(\C): k= \left\lceil \frac{\ell(\lambda^{1})}{2}\right \rceil   \right\}.$$
Furthermore, the column sums $2k^{2}-k$ and $2k^{2}+k$ count the number of vertices of  $B_{\lambda^{1}}$ and $B_{\lambda^{2}}$ respectively.
\end{theorem}
\begin{proof}
Since   $|{\lambda^{1}}|$ and  $|{\lambda^{2}}|$ are respectively in the odd and even positions of the sequence $\{{{k+1}\choose 2}\}_{k=1}$ of triangular numbers.  It is either both  $|{\lambda^{1}}|$ and  $|{\lambda^{2}}|$ are even or odd numbers, so there exists $k\in\N$ such that $|{\lambda^{1}}| :=2k^{2}-k$ and $|{\lambda^{2}}|:=2k^{2}+k$ are respectively  so that  if $k$ is odd so also $|{\lambda^{1}}|$ and $|{\lambda^{2}}|$, on the other hand if $k$ is even so also $|{\lambda^{1}}|$ and $|{\lambda^{2}}|$. These are  $(k+1)^{th}$ and $(k+2)^{th}$ terms of the sequence for $k\geq 2$. Hence the degree sharing  2-colour separations of $B_{\lambda^{1}}$ and $B_{\lambda^{2}}$ are  $C_{k^{2}, k^{2}-k}$ and $C_{k^{2}+k,k^{2}}$. 
\end{proof}
\begin{remark}
 The sequence $(2k^2-k, 2k^2+k)_{k=1}$ of total number of vertices corresponding to the parity pairs $(B_{\lambda^{1}}, B_{\lambda^{2}})$ is closely connected with the geometry of simplectic manifolds $Sp(k)$ in that each of the manifold is of dimension $2k^2+k$ in $\R^{4k^2}$ cut out by $2k^2-k$ equations.
 \end{remark}
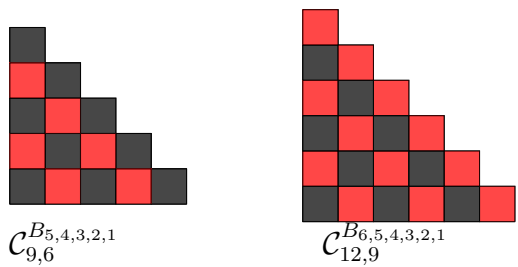
\begin{example}
 \[\begin{array}{lll}
	\vcenter{\hbox{\begin{tikzpicture}[scale=0.47]
			\draw (0,0)--(5,0);
			\draw (0,1)--(5,1);
			\draw (5,1)--(5,0);
			\draw (0,0)--(0,5);
			\draw (1,0)--(1,5);
			\draw (0,5)--(1,5);
			\draw (0,2)--(4,2);
			\draw (4,2)--(4,0);
			\draw (0,3)--(3,3)--(3,0);
			\draw (0,4)--(2,4)--(2,0);
			\filldraw[fill=black!90!white, draw=black, opacity=0.8] (0,5)--(1,5)--(1,4)--(0,4)--(0,5);
			\filldraw[fill=black!90!white, draw=black, opacity=0.8] (1,4)--(2,4)--(2,3)--(1,3)--(1,4);
			\filldraw[fill=black!90!white, draw=black, opacity=0.8] (2,3)--(3,3)--(3,2)--(2,2)--(2,3);
			\filldraw[fill=black!90!white, draw=black, opacity=0.8] (3,2)--(4,2)--(4,1)--(3,1)--(3,2);
			\filldraw[fill=black!90!white, draw=black, opacity=0.8] (4,1)--(5,1)--(5,0)--(4,0)--(4,1);
			\filldraw[fill=black!90!white, draw=black, opacity=0.8] (0,3)--(1,3)--(1,2)--(0,2)--(0,3);
			\filldraw[fill=black!90!white, draw=black, opacity=0.8] (1,2)--(2,2)--(2,1)--(1,1)--(1,2);
			\filldraw[fill=black!90!white, draw=black, opacity=0.8] (2,1)--(3,1)--(3,0)--(2,0)--(2,1);
			\filldraw[fill=black!90!white, draw=black, opacity=0.8] (0,1)--(1,1)--(1,0)--(0,0)--(0,1);
			\filldraw[fill=red!90!white, draw=black, opacity=0.8] (0,1)--(0,2)--(1,2)--(1,1)--(0,1);
			\filldraw[fill=red!90!white, draw=black, opacity=0.8] (1,1)--(2,1)--(2,0)--(1,0)--(1,1);
			\filldraw[fill=red!90!white, draw=black, opacity=0.8] (0,4)--(1,4)--(1,3)--(0,3)--(0,4);
			\filldraw[fill=red!90!white, draw=black, opacity=0.8] (1,3)--(2,3)--(2,2)--(1,2)--(1,3);
			\filldraw[fill=red!90!white, draw=black, opacity=0.8] (2,2)--(3,2)--(3,1)--(2,1)--(2,2);
			\filldraw[fill=red!90!white, draw=black, opacity=0.8] (3,1)--(4,1)--(4,0)--(3,0)--(3,1);
			\end{tikzpicture}}} &\qquad & \vcenter{\hbox{\begin{tikzpicture}[scale=0.47]
			\draw (0,0)--(5,0);
			\draw (0,1)--(5,1);
			\draw (5,1)--(5,0);
			\draw (0,0)--(0,5);
			\draw (1,0)--(1,5);
			\draw (0,5)--(1,5);
			\draw (0,2)--(4,2);
			\draw (4,2)--(4,0);
			\draw (0,3)--(3,3)--(3,0);
			\draw (0,4)--(2,4)--(2,0);
			\draw (6,0)--(6,1)--(5,1);
			\draw (5,1)--(5,2)--(4,2);
			\draw (4,2)--(4,3)--(3,3);
			\draw (3,3)--(3,4)--(2,4);
			\draw (2,4)--(2,5)--(1,5);
			\draw (1,5)--(1,6)--(0,6);
			\draw (0,5)--(0,6);
			\draw (5,0)--(6,0);
			\filldraw[fill=black!90!white, draw=black, opacity=0.8] (0,5)--(1,5)--(1,4)--(0,4)--(0,5);
			\filldraw[fill=red!90!white, draw=black, opacity=0.8] (0,6)--(1,6)--(1,5)--(0,5)--(0,6);
			\filldraw[fill=red!90!white, draw=black, opacity=0.8] (1,5)--(2,5)--(2,4)--(1,4)--(1,5);
			\filldraw[fill=red!90!white, draw=black, opacity=0.8] (2,4)--(3,4)--(3,3)--(2,3)--(2,4);
			\filldraw[fill=red!90!white, draw=black, opacity=0.8] (3,3)--(4,3)--(4,2)--(3,2)--(3,3);
			\filldraw[fill=red!90!white, draw=black, opacity=0.8] (4,2)--(5,2)--(5,1)--(4,1)--(4,2);
			\filldraw[fill=red!90!white, draw=black, opacity=0.8] (5,1)--(6,1)--(6,0)--(5,0)--(5,1);
			\filldraw[fill=red!90!white, draw=black, opacity=0.8] (0,4)--(1,4)--(1,3)--(0,3)--(0,4);
			\filldraw[fill=red!90!white, draw=black, opacity=0.8] (1,3)--(2,3)--(2,2)--(1,2)--(1,3);
			\filldraw[fill=red!90!white, draw=black, opacity=0.8] (2,2)--(3,2)--(3,1)--(2,1)--(2,2);
			\filldraw[fill=red!90!white, draw=black, opacity=0.8] (3,1)--(4,1)--(4,0)--(3,0)--(3,1);
			\filldraw[fill=red!90!white, draw=black, opacity=0.8] (0,2)--(1,2)--(1,1)--(0,1)--(0,2);
			\filldraw[fill=red!90!white, draw=black, opacity=0.8] (1,1)--(2,1)--(2,0)--(1,0)--(1,1);
			\filldraw[fill=black!90!white, draw=black, opacity=0.8] (1,4)--(2,4)--(2,3)--(1,3)--(1,4);
			\filldraw[fill=black!90!white, draw=black, opacity=0.8] (2,3)--(3,3)--(3,2)--(2,2)--(2,3);
			\filldraw[fill=black!90!white, draw=black, opacity=0.8] (3,2)--(4,2)--(4,1)--(3,1)--(3,2);
			\filldraw[fill=black!90!white, draw=black, opacity=0.8] (4,1)--(5,1)--(5,0)--(4,0)--(4,1);
			\filldraw[fill=black!90!white, draw=black, opacity=0.8] (0,3)--(1,3)--(1,2)--(0,2)--(0,3);
			\filldraw[fill=black!90!white, draw=black, opacity=0.8] (1,2)--(2,2)--(2,1)--(1,1)--(1,2);
			\filldraw[fill=black!90!white, draw=black, opacity=0.8] (2,1)--(3,1)--(3,0)--(2,0)--(2,1);
			\filldraw[fill=black!90!white, draw=black, opacity=0.8] (0,1)--(1,1)--(1,0)--(0,0)--(0,1);
			\end{tikzpicture}}} \\ 
			\mathcal{C}^{B_{5,4,3,2,1}}_{9, 6} &\qquad & \ \  \mathcal{C}^{B_{6,5,4,3,2,1}}_{12, 9}\\
                         \end{array}\]   
   \end{example}                                            

\section{ Equations of the 2-colour separation $\mathcal{C}^{B_{\lambda}}_{|\mu|, |\kappa|}$ }
\noindent Here we will discuss first the toric variety $\mathcal{V}(I_{\mathcal{C}^{B_{\lambda}}_{|\mu|, |\kappa|}})$ associated with the 2 colour separation  $\mathcal{C}^{B_{\lambda}}_{|\mu|, |\kappa|}$  of the pseudo-multipartite graph $B_{\lambda}$ in the affine space $\mathbb{A}^{\ell(\lambda)+2}$. We shall use the theory of  partition identity.  Fix a positive integer $n$. A partition identity is  given by
\begin{equation}
x_1 + x_2 + x_3+\cdots + x_r = y_1 + y_2+\cdots +y_d
\end{equation}
where $0< x_{i},y_{j}\leq n$ and all parts are generally not distinct. The identity $(4.1)$ is said to be primitive if it does not admit a proper subidentity
$$x_{a_1}+x_{a_2}+\cdots +x_{a_m} = y_{b_1}+y_{b_2}+\cdots +y_{b_n}$$
 where $1<m+n\leq r+d -1$. Interested reader should consult [Ch.6, sturmfels] .
\begin{lemma}
Let $\mathcal{C}^{B_{\lambda}}_{|\mu|, |\kappa|}$ be the 2- colour separation  identified with the  pseudo-multipartite graph $B_{\lambda}$ such that the length $\ell(\lambda)\geq 5$ and $\lambda = (\ell(\lambda), \ell(\lambda)-1,\dots, 2, 1)$.  Then the  parts of  partition identity
 $$1 + 2 + 3 + \cdots + \ell(\lambda) = |\mu| + |\kappa| $$  are all distinct.
\end{lemma}
\begin{proof}
Suppose that $\ell(\lambda)$ is odd, The sum of the black vertices is  $|\mu| = 1 + 3 + 5 +\cdots + \ell(\lambda)$ and that of the red vertices is $|\kappa|= 2 + 4 + 6 +\cdots + \ell(\lambda)-1$. Notice that the case is reversed if $\ell(\lambda)$ is even.  Notice that $|\mu| = \lceil \frac{\ell(\lambda)}{2}\rceil^{2}$ and since $\ell(\lambda)\geq 5$, $|\mu| > \ell(\lambda)$  and so $|\mu|$ cannot lie between $1$ and $\ell(\lambda)$. Likewise, there does not exist a positive integer $b$ such that $1\leq b \leq \ell(\lambda)$ and $b = |\kappa|$. To see this, notice that $|\kappa|= \lfloor \frac{\ell(\lambda)}{2}\rfloor(\lfloor \frac{\ell(\lambda)}{2}\rfloor + 1)$, therefore,  $|\kappa|> \ell(\lambda)$ for $\ell(\lambda)\geq 5$.  Since $\ell(\lambda)$ is odd,  $\lceil \frac{\ell(\lambda)}{2}\rceil > \lfloor \frac{\ell(\lambda)}{2}\rfloor$, so 
$\lceil \frac{\ell(\lambda)}{2}\rceil^{2} - \lceil \frac{\ell(\lambda)}{2}\rceil =\lfloor \frac{\ell(\lambda)}{2}\rfloor^{2} + \lfloor \frac{\ell(\lambda)}{2}\rfloor$, therefore, $|\mu| > |\kappa|$.\\ 
\noindent Suppose that $\ell(\lambda)$ is even, it suffices to prove that $|\mu|$ and $|\kappa|$ are distinct. Notice that  the red vertices are in the odd positions so the sum is $|\mu| =  \ell(\lambda)+ \ell(\lambda)-2 +\cdots + 4 + 2$ while that of black vertices is $|\kappa|= 1 + 3 + 5 +\cdots + \ell(\lambda)-1$, so  $|\mu|= \lfloor \frac{\ell(\lambda)}{2}\rfloor(\lfloor \frac{\ell(\lambda)}{2}\rfloor + 1)$   and  $|\kappa| = \lceil \frac{\ell(\lambda)}{2}\rceil^{2}$.  Since $\ell(\lambda)$ is even, $\lceil \frac{\ell(\lambda)}{2}\rceil = \lfloor \frac{\ell(\lambda)}{2}\rfloor$, hence, $|\mu| > |\kappa|$.
\end{proof}
\noindent The partition identity 
\begin{equation}
1 + 2 + 3 + \cdots + \ell(\lambda) = |\mu| + |\kappa| 
\end{equation}
associated with the 2-colour separation $\mathcal{C}^{B_{\lambda}}_{|\mu|, |\kappa|}$ is called colour separation partition identity for $\mathcal{C}^{B_{\lambda}}_{|\mu|, |\kappa|}$ (cspi for short). Notice that the identity (4.2) is not primitive, in fact, it is  obvious from the Lemma 4.1 that every cspi admits two proper subidentities which are primitive. These subidentities have been  characterised  already in $(3.1)$ and $(3.2)$. They turn out to give the characterisation of the toric ideal associated with the 2-colour separation $\mathcal{C}^{B_{\lambda}}_{|\mu|, |\kappa|}$ for $\ell(\lambda)\geq 5$. Let $\mathcal{A}\in \Z^{m\times n}$ be the $m\times n$ integer matrix. It is well known that the toric ideal $\mathcal{I}_{\mathcal{A}}$ is generated by the Graver basis of $\mathcal{A}$. This is very important in what follows.\\
\ \\
\noindent {\bf Conjecture I:} \emph{ Let $d=1$ and $\mathcal{A}$ be the $1\times \ell(\lambda)+2$ integer matrix $\mathcal{A} =(1\ 2 \cdots \ \ell(\lambda) \ |\mu| \  |\kappa| )\in\Z^{1\times \ell(\lambda)+2}$. Then for $\ell(\lambda)\geq 5$, the binomial ideal $\mathcal{I}_{\mathcal{A}}$ associated with the 2-colour separation $\mathcal{C}^{B_{\lambda}}_{|\mu|, |\kappa|}$ of the pseudo-multipartite graph $B_{\lambda}$ is of degree $\lceil \frac{\ell(\lambda)}{2}\rceil \cdot\lfloor \frac{\ell(\lambda)}{2}\rfloor$, dimension $\ell(\lambda)$ and  minimally generated. That is},

$$\mathcal{I}_{\mathcal{A}} = \left\{\begin{array}{ll}   \langle x_{1}x_{3}x_{5}\cdots x_{\ell(\lambda)} -x_{|\mu|},  x_{2}x_{4}x_{6}\cdots x_{\ell(\lambda)-1} -x_{|\kappa|} \rangle, \  \mbox{if}\; \  \ell(\lambda) \  {\rm is \  odd} \\
\ \\
\langle  x_{1}x_{3}x_{5}\cdots x_{\ell(\lambda)-1} -x_{|\kappa|}, x_{2}x_{4}x_{6}\cdots x_{\ell(\lambda)} -x_{|\mu|}\rangle, \;  \mbox{if}\; \  \ell(\lambda) \ {\rm is \  even} \end{array}\right.$$ 
\ \\
\begin{remark}
For $d=1$ and $\mathcal{A} =(1\ 2 \cdots \ \ell(\lambda) \ |\mu| \  |\kappa| )\in\Z^{1\times \ell(\lambda)+2}$. Let $$\C[x_1,x_2,\dots, x_{\ell(\lambda)},x_{|\mu|},x_{|\kappa|}]\longrightarrow \C[t] \ {\rm be \ defined \ by} \ x_i\mapsto t^i.$$  The  binomial ideal $\mathcal{I}_{\mathcal{A}}$  generated by $x_{1}x_{3}x_{5}\cdots x_{\ell(\lambda)} -x_{|\mu|}$ and $x_{2}x_{4}x_{6}\cdots x_{\ell(\lambda)-1} -x_{|\kappa|}$
({\rm resp} $x_{1}x_{3}x_{5}\cdots x_{\ell(\lambda)-1} -x_{|\kappa|}$ and $x_{2}x_{4}x_{6}\cdots x_{\ell(\lambda)} -x_{|\mu|}$) when $\ell(\lambda)$ is odd (resp even) minimally generates the kernel. So, there is a correspondence between  elements of Grava basis Gr($\mathcal{A}$) and the two proper primitive subidentities of 
 $1 + 2 + 3 + \cdots + \ell(\lambda) = |\mu| + |\kappa|. $
 \end{remark}
 \ \\
 \noindent The projective version of the toric variety is given by the cartoon diagram of the 2-colour separation  identified with $B_{\lambda}$.  The diagram is realised from the fact that for $3\leq k\leq \ell(\lambda)$, the disjoint vertex subsets $V_{k}$ and $V_{k-1}$ share  different colourings and there are exactly $2(k-1)$ edges between them while $V_{k}$ and $V_{k-2}$ are disconnected because they share the same colour. The general cartoon is illustrated in Figure 6. The cardinalities of the disjoint vertex subsets constitute the weights of the nodes  and the directed edges connote the two-way interaction among them. Notice that the empty subset is added for stability. 
 
\begin{figure}[!hbt] 
 \begin{center}
\begin{tikzpicture}
\draw[->]  (0,0)--(0,1);
\draw[fill] (0,0) circle [radius=0.035];
\draw (0,1)--(0,2);
\draw[fill] (0,2) circle [radius=0.035];
\draw [->] (0,2)--(1,2);
\draw (1,2)--(2,2);
\draw[fill] (2,2) circle [radius=0.035];
\draw [-> ](4,2)--(5,2);
\draw  (5,2)--(6,2);
\draw[fill] (6,2) circle [radius=0.035];
\draw(6,1)--(6,2);
\draw[fill] (4,2) circle [radius=0.035];
\draw  (6,0)--(6,1);
\draw (4,1)--(4,2);
\draw[fill] (6,0) circle [radius=0.035];
\draw [->](4,0)--(4,1);
\draw[fill] (4,0) circle [radius=0.035];
\draw[->] (2,0)--(2,1);
\draw (2,1)--(2,2);
\draw[fill] (2,0) circle [radius=0.035];
\node at (0,2.3){$\ell(\lambda)$};
\node at (0,-0.3){$\ell(\lambda)-1$};
\node at (2,2.3){$\ell(\lambda)-1$};
\node at (2,-0.3){$\ell(\lambda)-2$};
\node at (4,2.3){$2$};
\node at (4,-0.3){$1$};
\node at (6,2.3){$1$};
\node at (6,-0.3){$0$};

\node at (3,2){$-$};
\node at (2.5,2){$-$};
\node at (3.5,2){$-$};
\end{tikzpicture}
\caption{The general cartoon of the 2-colour separation $\mathcal{C}^{B_{\lambda}}_{|\mu|, |\kappa|}$ }
\end{center}
\end{figure}
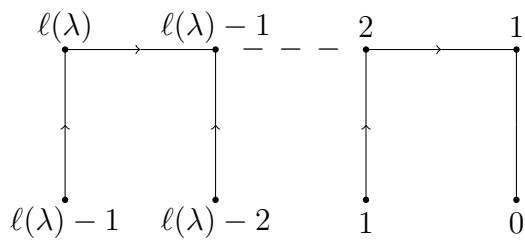

\noindent Let $\C[x_{i}: 0\leq i\leq \ell(\lambda)]$ be the free algebra  on the weights of the nodes of the cartoon associated with the 2-colour separation $\mathcal{C}^{B_{\lambda}}_{|\mu|, |\kappa|}$. Define the map $$\C[x_{i}: 0\leq i\leq \ell(\lambda)]\longrightarrow \C[t],\  {\rm as} \ x_i\mapsto t^{i},$$  
 so that the binomial ideal $\mathcal{I}_{\mathcal{C}^{B_{\lambda}}_{|\mu|, |\kappa|}}$ parameterizing the colour segregation of the cartoon diagram of $\mathcal{C}^{B_{\lambda}}_{|\mu|, |\kappa|}$ is given by 
 \begin{equation}
 \mathcal{I}_{\mathcal{C}^{B_{\lambda}}_{|\mu|, |\kappa|}}=\langle x_{\ell(\lambda)-2-i}x_{\ell(\lambda)-i}-x^{2}_{\ell(\lambda)-1-i} : 0\leq i\leq \ell(\lambda)-2 \rangle.
 \end{equation}
 We thus have the algebra 
 \begin{equation}
 R_{\mathcal{C}^{B_{\lambda}}_{|\mu|, |\kappa|}}\cong \C[x_{i}: 0\leq i\leq \ell(\lambda)]/\langle x_{\ell(\lambda)-2-i}x_{\ell(\lambda)-i}-x^{2}_{\ell(\lambda)-1-i} : 0\leq i\leq \ell(\lambda)-2 \rangle.
 \end{equation}
 The corresponding toric embedding is
 \begin{equation}
 \phi : X_{\mathcal{C}^{B_{\lambda}}_{|\mu|, |\kappa|}} \hookrightarrow {\rm Proj}(\C[x_{i}: 0\leq i\leq \ell(\lambda)])\cong \mathbb{P}^{\ell(\lambda)}
 \end{equation}

\noindent{\bf Conjecture 2:} \emph{The binomial ideal $\mathcal{I}_{\mathcal{C}^{B_{\lambda}}_{|\mu|, |\kappa|}}$ of the embedding of $X_{\mathcal{C}^{B_{\lambda}}_{|\mu|, |\kappa|}}$ into $\mathbb{P}^{\ell(\lambda)}$ via cartoon diagram is 2-dimensional, generated by $\ell(\lambda)-1$ quadrics and of degree $2^{\ell(\lambda)-1}$.} \\

\noindent  For instance, the cartoon diagram for the 2-colour separation  $\mathcal{C}^{B_{3,2,1}}_{4, 2}$ is illustrated in the Figure 8,

 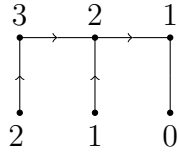
\begin{figure}[!hbt]
\begin{center}
\begin{tikzpicture}
\draw[->]  (0,0)--(0,0.5);
\draw[fill] (0,0) circle [radius=0.035];
\draw (0,0.5)--(0,1);
\draw[fill] (0,1) circle [radius=0.035];
\draw [->] (0,1)--(0.5,1);
\draw (0.5,1)--(1,1);
\draw[fill] (1,1) circle [radius=0.035];
\draw[->]  (1,1)--(1.5,1);
\draw (1.5,1)--(2,1);
\draw[fill] (2,1) circle [radius=0.035];
\draw  (2,1)--(2,0.5);
\draw[fill] (2,0) circle [radius=0.035];
\draw  (2,0.5)--(2,0);
\draw[fill] (1,0) circle [radius=0.035];
\draw (1,0.5)--(1,1);
\draw [->] (1,0)--(1,0.5);

\node at (0,1.3){$3$};
\node[left] at (0.2,-0.3){$2$};
\node  at (1,1.3){$2$};
\node at (2,1.3){$1$};
\node at (1,-0.3){$1$};
\node at (2,-0.3){$0$};
\end{tikzpicture}
\caption{The cartoon diagram of the 2-colour separation $\mathcal{C}^{B_{3,2,1}}_{4, 2}$}
\end{center}
\end{figure}

 Therefore, the ideal is given the equation (4.4).

 \begin{equation}
 \mathcal{I}_{\mathcal{C}^{B_{3,2,1}}_{4, 2}} =\langle x_{0}x_{2}-x_{1}^{2}, x_{1}x_{3}-x_{2}^{2}\rangle.
 \end{equation}

\begin{center}
{\bf References}
\end{center}
\begin{enumerate}
\item [{[1.]}] S. Assaf and A. Schilling, \emph{A Demazure crystal construction for Schubert polynomial} Volume 1, issue 2 (2018), p. 225-247.
\item[{[2.]}] H. Ohsugi and T. Hibi, \emph{Compressed polytopes, initial ideals and complete multipartite graphs,} Illinois
J. Math. 44, No. 2 (2000), 391-406.
\item [{[3.]}] A. Higashitani and K. Matsushita,\emph{ Conic divisorial ideals and non-commutative crepant resolutions of edge rings of complete multipartite graphs,} arXiv:2011.07714.
\item [{[4.]}] B. Davison, J. Ongaro and B. Szendroi  K,\emph{ Enumerating coloured partitions in 2 and 3 dimensions} arXiv:1811.12857.
\item[{[5.]}] A. Higashitani and K. Matsushita,\emph{Levelness versus almost Gorensteinness of edge rings of complete multipartite graphs,} arXiv:2102.02349v1.
\item[{[6.]}]  A. Gainer-Dewar and Ira M. Gessel, \emph{Enumeration of bipartite graphs and bipartite blocks}, arXiv:1304.0139v2.
\item [{[7.]}]T. Lam, \emph{Affine Stanley symmetric functions}, Amer. J. Math. 128 (2006), no. 6, 1553-1586.
\item[{[8.]}] B.E Sagan, \emph{The Symmetric Group: Representations, Combinatorial Algorithms, and Symmetric Functions}, (2nd Ed.), (Springer, 2013).
\item [{[9.]}]R. Stanley, \emph{On the number of reduced decompositions of elements of Coxeter groups}, European J. Combinatorics 5(1984)
   359--509.
\item[{[10.]}] B. Sturmfels, \emph{Gr\"{o}bner Bases and Convex Polytopes}. University Lecture Series, AMS.  Vol. 8. 1991.  
\item[{[11.]}] Guo, Z., Xiao, M., Zhou, Y. (2020). \emph{The Complexity of the Partition Coloring Problem.} In: Chen, J., Feng, Q., Xu, J. (eds) Theory and Applications of Models of Computation. TAMC 2020. Lecture Notes in Computer Science(), vol 12337. Springer.
\item[{[12.]}] V.  Yegnanarayanan, \emph{Graph colouring and Partitions},  Theortical computer science 263 (2001) 59-74.
\end{enumerate}

\end{document}